\documentclass[12pt]{article}

\usepackage{amssymb,amsfonts,amsthm}
\usepackage[centertags]{amsmath}
\usepackage[english]{babel}
\usepackage{color}

\newtheorem{theorem}{Theorem}[section]

\newtheorem{conjecture}[theorem]{Conjecture}
\newtheorem{corollary}[theorem]{Corollary}

\newtheorem{example}[theorem]{Example}

\newtheorem{lemma} [theorem]{Lemma}

\newtheorem{problem}[theorem]{Problem}
\newtheorem{proposition}[theorem]{Proposition}

\renewenvironment{proof}[1][Proof]{\textbf{#1.} }{\ \rule{0.5em}{0.5em}}
\newtheorem{observation}[theorem]{Observation}

\title{Changing of the domination number of a graph: edge multisubdivision and edge removal} 
\author{ Vladimir Samodivkin \\
Department of Mathematics \\
University of Architecture Civil Engineering
and Geodesy\\
Hristo Smirnenski 1 Blv., 1046 Sofia, Bulgaria,\\
 \texttt{vlsam\_fte@uacg.bg}
 }

\begin{document}
\maketitle

\newtheorem{them}{Theorem}
\newtheorem{lema}[them]{Lemma}
\renewcommand{\thethem}{ \Alph{them}}
\renewcommand{\thelema}{\Alph{lema}}



\bigskip

 \begin{abstract}
For a graphical property $\mathcal{P}$ and a graph $G$,
 a subset $S$ of vertices of $G$ is a $\mathcal{P}$-set 
 if the subgraph induced by $S$ has the property $\mathcal{P}$. 
 The domination number with respect to the property $\mathcal{P}$,
  denoted by $\gamma_{\mathcal{P}} (G)$,  
  is the minimum cardinality of a dominating $\mathcal{P}$-set. 
   We define the domination multisubdivision number with respect to $\mathcal{P}$,  
denoted by  $msd_{\mathcal{P}}(G)$, as a minimum positive integer $k$ 
such that there exists an edge which must be subdivided $k$ times to change $\gamma_\mathcal{P} (G)$.
    In this paper  
	(a) 	we present necessary and sufficient conditions 
	for a change  of $\gamma_{\mathcal{P}}(G)$ after subdividing an edge of $G$ once, 
   (b) we prove that  if $e$ is an edge   of a graph $G$  then 
	 $\gamma_\mathcal{P} (G_{e,1}) < \gamma_\mathcal{P} (G)$  if and only if 
	$\gamma_\mathcal{P} (G-e) < \gamma_\mathcal{P} (G)$ 
($G_{e,t}$ denote the graph obtained from $G$ by subdivision of  $e$ with $t$ vertices), 
	(c) we also prove that for every edge of a graph $G$ is fulfilled  
$\gamma_{\mathcal{P}}(G-e) \leq \gamma_{\mathcal{P}}(G_{e,3}) \leq  \gamma_{\mathcal{P}}(G-e) + 1$, 
and 
(d)  we show  that $msd_{\mathcal{P}}(G) \leq 3$, 
	where $\mathcal{P}$ is hereditary and closed under union with $K_1$. 
 \end{abstract}
{\bf Keywords} {dominating set, edge subdivision, domination multisubdivision \\ number,   
hereditary graph property}


\section{Introduction}
All graphs considered in this article are finite, undirected, without loops or multiple edges. 
 For the graph theory terminology not presented here, we follow Haynes et al.   \cite{hhs1}. 
We denote the vertex set and the edge set of a graph $G$ by $V(G)$ and $ E(G),$  respectively.  
 The subgraph induced by $S \subseteq V(G)$ is denoted by $\left\langle S,G  \right\rangle$. 
 For a vertex $x$ of $G$,  $N(x,G)$ denotes the set of all  neighbors of $x$ in $G$,   $N[x,G] = N(x,G) \cup \{x\}$
and the degree of $x$ is $deg(x,G) = |N(x,G)|$. The maximum and  minimum  degrees of vertices in the graph $G$ are denoted by $\Delta (G)$ and $\delta (G)$ respectively. 
  For a graph $G$, let $x \in X \subseteq V(G)$. A vertex $y$ is
	a {\it private neighbor  of $x$ with respect to} $X$ if $N[y, G] \cap X = \{x\}$. 
	The {\it private neighbor set of $x$ with respect to} $X$ is $pn_G[x,X] = \{y$  : $N[y,G] \cap X = \{x\}\}$. 
	For a graph $G$ subdivision of the edge $e=uv\in E(G)$ with vertex $x$ leads to a graph with vertex set $V\cup \{x\}$ and edge set $(E-\{uv\})\cup \{ux,xv\}$. 
Let $G_{e,t}$ denote graph obtained from $G$ by subdivision of the edge $e$ with $t$ vertices (instead of edge $e=uv$ we put a path $(u,x_1,x_2,\ldots ,x_t,v)$). For $t=1$ we write $G_e$.

Let $\mathcal{I}$ denote the set of all mutually non-isomorphic graphs. A \emph{ graph property} is any non-empty subset of $\mathcal{I} $. We say that a \emph{ graph $G$ has the  property} $\mathcal{P}$ whenever there exists a graph $H \in \mathcal{P} $ which is isomorphic to $G$. 
For example, we list some graph properties: 

$\bullet$    $\mathcal{O} = \{H \in \mathcal{I}$ : $H$ is totally disconnected$\}$;

$\bullet$      $\mathcal{C} = \{H \in \mathcal{I}$ : $H$ is connected$\}$;
   
$\bullet$        $\mathcal{T} = \{H \in \mathcal{I}$ : $\delta (H) \geq 1$ $\}$;

$\bullet$        $\mathcal{F} = \{H \in \mathcal{I}$ : $H$ is a forest$\}$;
   
$\bullet$       $\mathcal{UK} = \{H \in \mathcal{I}$ :  each component of $H$ is complete$\}$;

$\bullet$       $\mathcal{D}_k = \{H \in \mathcal{I}$ :  $ \Delta (H) \leq k\}$.

A graph property $\mathcal{P}$ is called: (a) hereditary (induced-hereditary), 
if from the fact that a graph $G$ has property $\mathcal{P}$,
 it follows that all  subgraphs (induced subgraphs) of $G$ also belong to $\mathcal{P}$, and  
 (b)  nondegenerate if $\mathcal{O}  \subseteq \mathcal{P}$.
  Any  set $S \subseteq V(G)$ such that the induced subgraph $\left\langle S,G  \right\rangle$  
	possesses the property $\mathcal{P}$ is called a $\mathcal{P}$- set.  
	 Note that: (a) $\mathcal{I}$, $\mathcal{F}$ and $\mathcal{D}_k$ are nondegenerate and hereditary  properties; 
(b)  $\mathcal{UK}$ is nondegenerate,  induced-hereditary and is not hereditary;
(c) both $\mathcal{C}$ and $\mathcal{T}$ are neither  induced-hereditary nor nondegenerate. 
	  For a survey on this subject we refer to Borowiecki et al. \cite{bbfms}.

  A   set of vertices $D \subseteq V(G)$ is a  dominating set of a graph $G$  
if every vertex not in $D$ is adjacent to a vertex in $D$. 
   The  domination number with respect to the property $\mathcal{P}$, denoted by $\gamma_{\mathcal{P}} (G)$,
	is the smallest cardinality of a dominating $\mathcal{P}$-set of $G$. 
   A dominating $\mathcal{P}$-set of $G$ with cardinality $\gamma_{\mathcal{P}} (G)$ is called a 
   $\gamma_{\mathcal{P}}$-{\it set of} $G$. 
   If a property $\mathcal{P}$ is nondegenerate, then every maximal independent set is a $\mathcal{P}$-set and thus $\gamma_{\mathcal{P}} (G)$ exists. 
   Note that $\gamma_{\mathcal{I}} (G)$, $\gamma_{\mathcal{O}} (G)$, $\gamma_{\mathcal{C}} (G)$, $\gamma_{\mathcal{T}} (G)$,$\gamma_{\mathcal{F}} (G)$ and $\gamma_{\mathcal{D}_k} (G)$, are the 
 well known as the domination number $\gamma (G)$, the independent domination number $i(G)$, 
 the connected domination number $\gamma_c (G)$, the total domination number $\gamma_t (G)$, 
the acyclic domination number $\gamma_a (G)$  and the $k$-dependent domination number $\gamma^k (G)$,
respectively (see \cite{hhs1}). 
 The concept of domination with respect to any graph property $\mathcal{P}$
  was introduced by Goddard et al. \cite{ghk} and has been studied, for example, 
	in  \cite{m,sam, sam3, sam4, sam5} and elsewhere.    
	
It is often of interest to know how the value of a graph parameter is affected when a small change is made in a graph.  
In \cite{sam}, the present author began the study of the effects on 
 $\gamma_\mathcal{P}(G)$ when a graph $G$ is modified by deleting a vertex or by adding an edge 
($\mathcal{P}$ is  nondegenerate).   
 In this paper we concentrate on effects on $\gamma_{\mathcal{P}} (G)$ when a graph is modified 
 by  deleting/subdividing an edge. 
  An edge $e$ of a graph $G$ is called a $\gamma_{\mathcal{P}}$-$ER^-$-critical edge of $G$  
	if $\gamma_{\mathcal{P}}(G) > \gamma_{\mathcal{P}}(G-e)$. 
	Note that (a) there do not exist $\gamma$-$ER^-$-critical edges (see \cite{gm}), 
	(b) Grobler \cite{g}  was the first who began the investigation of  $\gamma_{\mathcal{P}}$-$ER^-$-critical edges
	      when $\mathcal{P} = \mathcal{O}$, and 	
(c) necessary and sufficient conditions for an edge of a graph $G$ 
      to be $\gamma_{\mathcal{H}}$-$ER^-$-critical may be found in \cite{sam}.


When  an edge of a graph $G$ is subdivided,  the  domination number with respect to the property $\mathcal{P}$ can increase or decrease. 
For instance, if $G$ is a star $K_{1,p}$, $p \geq 2$, and $\{K_1, 2K_1\} \subseteq \mathcal{P} \subseteq \mathcal{I}$
 then $\gamma_{\mathcal{P}} (G) = 1$ and $\gamma_{\mathcal{P}} (G_e) =2$ for all $e$. 
If a graph $G$ is obtained by three stars $K_{1,p}$ and three edges $e_1, e_2, e_3$ joining their centers
 then $\gamma_{\mathcal{F}} (G) = 2+p$ and $\gamma_{\mathcal{F}} (G_{e_i}) =3$, $i=1,2,3$.
This motivates the next definitions.

For any  nondegenerate property $\mathcal{P} \subseteq \mathcal{I}$  we define the edge $e$ of a graph $G$ to be 
(a) a $\gamma_{\mathcal{P}}$-$S^+$-critical edge of $G$  if $\gamma_{\mathcal{P}}(G) < \gamma_{\mathcal{P}}(G_e)$, 
and 
(b) a $\gamma_{\mathcal{P}}$-$S^-$-critical edge of $G$  if $\gamma_{\mathcal{P}}(G) >\gamma_{\mathcal{P}}(G_e)$.
In Section 2:  (a) 	we present necessary and sufficient conditions 
	for a change  of $\gamma_{\mathcal{P}}(G)$ after subdividing an edge of $G$ once,
 and 
   (b) we prove that  an edge  $e$ of a graph $G$ is $\gamma_{\mathcal{H}}$-S$^-$-critical if and only if 
      $e$ is $\gamma_{\mathcal{H}}$-$ER^-$-critical, where $\mathcal{H} \subseteq \mathcal{I}$ 
      is any  induced-hereditary and closed under union with $K_1$ graph property. 
			
			In Section 3 we deals with  changing of $\gamma_{\mathcal{P}}(G)$ when 
			an edge of $G$ is multiple subdivided. 
			To present our results we need the following definitions. 
			
For every edge $e$ of a graph $G$ let 

$\bullet$  $msd_{\mathcal{P}}(e) = \min\{t \mid \gamma_\mathcal{P}(G_{e,t}) \not= \gamma_\mathcal{P}(G)\}$;  

$\bullet$  $msd_{\mathcal{P}}^+(e) = \min\{t \mid \gamma_\mathcal{P}(G_{e,t})  > \gamma_\mathcal{P}(G)\}$;

$\bullet$  $msd_{\mathcal{P}}^-(e) = \min\{t \mid \gamma_\mathcal{P}(G_{e,t})  < \gamma_\mathcal{P}(G)\}$.

If $\gamma_\mathcal{P}(G_{e,t})  \geq \gamma_\mathcal{P}(G)$  for every $t$, $t \geq 1$, then 
we write $msd_{\mathcal{P}}^-(e) =  \infty$. 
If $\gamma_\mathcal{P}(G_{e,t})  \leq \gamma_\mathcal{P}(G)$  for every $t$, $t \geq 1$, then 
we write $msd_{\mathcal{P}}^+(e) =  \infty$.

For every graph $G$ with at least one edge and every nondegenerate property $\mathcal{P}$, we define:

$\mathbf{(D_1)}$ the domination multisubdivision (plus domination multisubdivision, minus domination multisubdivision) number  
with respect to the property $\mathcal{P}$, denoted 
$msd_{\mathcal{P}}(G)$  ($msd_{\mathcal{P}}^+, msd_{\mathcal{P}}^-(G)$, respectively) 
to be 

$\bullet$  $msd_{\mathcal{P}}(G) = \min\{msd_{\mathcal{P}}(e) \mid e \in E(G)\}$, 

$\bullet$ $msd_{\mathcal{P}}^+(G) = \min\{msd_{\mathcal{P}}^+(e) \mid e \in E(G)\}$, and 

$\bullet$ $msd_{\mathcal{P}}^-(G) = \min\{msd_{\mathcal{P}}^-(e) \mid e \in E(G)\}$,
\\
 respectively.
If $\gamma_\mathcal{P}(G_{e,t})  \geq \gamma_\mathcal{P}(G)$  for every $t$, and every edge $e \in E(G)$, then 
we write $msd_{\mathcal{P}}^-(G)  =  \infty$. 

The parameters $msd^+(G)$ and $msd^+_\mathcal{T}(G)$ (in our designation) 
was introduced by   Dettlaff, Raczek and  Topp  in \cite{drt} and 
 by Avella-Alaminos,  Dettlaff, Lema\'nska and   Zuazua \cite{adlz}, respectively. 
Note that in the case when $\mathcal{P} = \mathcal{I}$, clearly, $msd(G) = msd^+(G) $, 
 and $msd^-(G) =  \infty$.   
 In Section 3 we prove that for every edge of a graph $G$ is fulfilled  
$\gamma_{\mathcal{P}}(G-e) \leq \gamma_{\mathcal{P}}(G_{e,3}) \leq  \gamma_{\mathcal{P}}(G-e) + 1$  
 and we present  necessary and sufficient conditions for the validity  of 
$\gamma_{\mathcal{P}}(G-e)= \gamma_{\mathcal{P}}(G_{e,3})$. 
Our main result in this section is  that $msd_{\mathcal{P}}(G) \leq 3$ for any graph $G$ 
and any hereditary and closed under union with $K_1$ graph-property $\mathcal{P}$.

\section{Single subdivision: critical edges}

We begin this section with a characterization of $\gamma_{\mathcal{P}}$-$S^+$-critical edges of a graph. 
 Note that if a property $\mathcal{P}$ is induced-hereditary and closed under union with $K_1$  then $\mathcal{P}$ is nondegenerate.

\begin{theorem} \label{plus1}
Let $\mathcal{H} \subseteq \mathcal{I}$ be 
hereditary and closed under union with $K_1$. 
Let $G$ be a graph and $e = uv \in E(G)$. 
Then $\gamma_\mathcal{H} (G_e) \leq \gamma_\mathcal{H} (G) +1$. 
If $e$ is  a $\gamma_{\mathcal{H}}$-$S^+$-critical edge of $G$ 
then $\gamma_\mathcal{H} (G_e) = \gamma_\mathcal{H} (G) +1$ and 
for each $\gamma_\mathcal{H}$-set $M$ of $G$ one of the following holds:
\begin{itemize}
\item[(i)] $u,v \in V(G) -M$; 
\item[(ii)] $u\in M$, $v \in pn_G[u,M]$ and $pn_G[u,M]$ is not a subset of $\{u, v\}$; 
\item[(iii)] $v\in M$, $u \in pn_G[v,M]$ and $pn_G[u,M]$ is not a subset of $\{u, v\}$.
\end{itemize}
If $e$ is not $\gamma_{\mathcal{P}}$-$S^+$-critical 
 and for each $\gamma_\mathcal{H}$-set $M$ of $G$ one of (i), (ii) and (iii) holds 
then there is a dominating $\mathcal{H}$-set $R$ of $G-uv$ with $u, v \in R$ and 
$|R| \leq \gamma_\mathcal{H} (G)$.
\end{theorem}
\begin{proof}
Let  $x \in V(G_e)$ be the subdivision vertex and let $M$ be a $\gamma_\mathcal{H}$-set of $G$. 
If  $u, v \not \in M$ then $M \cup \{x\}$ is a  dominating $\mathcal{H}$-set of $G_e$ 
($\mathcal{H}$ is closed under union with $K_1$)  
and we have $\gamma_\mathcal{H} (G_e) \leq \gamma_\mathcal{H} (G) + 1$. 
If both $u$ and $v$ are in $M$ then $M$ is a
dominating $\mathcal{H}$-set of $G_e$($\mathcal{H}$ is hereditary) which implies  
$\gamma_\mathcal{H} (G_e) \leq \gamma_\mathcal{H} (G)$.
 If $u \in M$, $v \not\in M$ and $v \not \in pn_G[u,M]$ 
 then again $M$ is a  dominating $\mathcal{H}$-set of $G_e$ 
and hence $\gamma_\mathcal{H} (G_e) \leq \gamma_\mathcal{H} (G)$.
 So, let $u\in M$, $v \not\in M$ and $v \in pn_G[u,M]$. 
 If either $\{v\}$ or $\{u, v\}$ coincides with $pn_G[u,M]$ then 
 $(M - \{u\}) \cup \{x\}$ is a  dominating $\mathcal{H}$-set of $G_e$; 
 hence $\gamma_\mathcal{H} (G_e) \leq \gamma_\mathcal{H} (G)$.
 If neither $pn_G[u,M] = \{v\}$ nor $pn_G[u,M] = \{u, v\}$ then 
 $M \cup \{v\}$ is a  dominating $\mathcal{H}$-set of $G_e$ 
and we have $\gamma_\mathcal{H} (G_e) \leq \gamma_\mathcal{H} (G) + 1$.
 Thus $\gamma_\mathcal{H} (G_e) \leq \gamma_\mathcal{H} (G) +1$ and if the equality is fulfilled then one of (i), (ii) and (iii) holds.

Now, let for each $\gamma_\mathcal{H}$-set $M$ of $G$ one of (i), (ii) and (iii) hold.   
Assume $\gamma_\mathcal{H} (G_e) \leq \gamma_\mathcal{H} (G)$ and let  $R$ be a $\gamma_\mathcal{H}$-set of $G_e$.  

{\it Case} 1: $u, v \not \in R$. Hence $x \in R$. 
If $u, v \not \in pn_{G_e}[x, R]$ then $R - \{x\}$ is a dominating $\mathcal{H}$-set of $G$,
 a contradiction with $\gamma_\mathcal{H} (G_e) \leq \gamma_\mathcal{H} (G)$. 
If  $u \in pn_{G_e}[x, R]$ and $v \not \in pn_{G_e}[x, R]$  then $R_1 = (R-\{x\}) \cup \{u\}$
 is a dominating $\mathcal{H}$-set of $G$ of cardinality $|R_1| = |R| = \gamma_\mathcal{H} (G_e)$. 
 Since $\gamma_\mathcal{H} (G_e) \leq \gamma_\mathcal{H} (G)$, 
we have that $R_1$ is a $\gamma_\mathcal{H}$-set of $G$. 
 But then $u \in R_1$, $v \not \in R_1$ and  $v \not \in pn_G[u, R_1]$, contradicting (ii). 
 If $u, v \in pn_G[x, R]$ then as above $R_1$ is a $\gamma_\mathcal{H}$-set of $G$  
 and since $u \in R_1$ and  $\{u, v\} = pn_G[u, R_1]$, again we arrive to a contradiction with (ii). 
 
{\it Case} 2: $u \in R$ and $v \not \in R$. Hence $x \not \in R$, otherwise 
              $R - \{x\}$ is a dominating $\mathcal{H}$-set of $G$, contradicting   
              $\gamma_\mathcal{H} (G_e) \leq \gamma_\mathcal{H} (G)$.
              This implies that $R$ is a $\gamma_\mathcal{H}$-set  of $G$,
               $u \in R$ and   $v \not \in pn_G[u, R]$, a contradiction with (ii). 

{\it Case} 3: $u, v \in R$. Hence $R$ is a dominating $\mathcal{H}$-set of $G-uv$ 
              and $|R| = \gamma_\mathcal{H} (G_e) \leq \gamma_\mathcal{H} (G)$.
\end{proof}

When we restrict our attention to the case where $\mathcal{H} = \mathcal{I}$, 
 we can describe more precisely when an edge of a graph $G$ is $\gamma$-S$^+$-critical.

\begin{corollary}
Let $G$ be a graph and $e = uv \in E(G)$. 
 Then $e$ is  a $\gamma$-S$^+$-critical edge of $G$ 
 if and only if for each $\gamma$-set $M$ of $G$  one of (i), (ii) and (iii) stated in Theorem \ref{plus1} holds.
\end{corollary}
\begin{proof}
{\it Necessity}: 
 The result immediately follows by Theorem \ref{plus1}.

 {\it Sufficiency}: 
Assume $\gamma (G_e) \leq \gamma (G)$. 
Then by Theorem \ref{plus1},  there is a dominating set $R$ of $G-uv$ with $u, v \in R$ and 
$|R| \leq \gamma(G)$. But it is well known fact that if $f$ is an edge of a graph $G$ then always $\gamma (G-f) \geq \gamma (G)$. 
Hence $R$ is a $\gamma$-set of both $G$ and  $G-e$ and  $u, v \in R$,  contradicting all (i), (ii) and (iii).
\end{proof}

\begin{theorem} \label{minus}
Let $\mathcal{H} \subseteq \mathcal{I}$ be induced-hereditary and closed under union with $K_1$. 
An edge  $e$ of a graph $G$ is $\gamma_{\mathcal{H}}$-S$^-$-critical if and only if 
$e$ is $\gamma_{\mathcal{H}}$-$ER^-$-critical.  
\end{theorem}
\begin{proof}
 As we have already know, $\mathcal{H}$ is nondegenerate and then all 
 $\gamma_\mathcal{H} (G-e)$, $\gamma_\mathcal{H} (G_e)$ 
and $\gamma_\mathcal{H} (G)$ exist. Let $v$ be the subdivision vertex of $G_e$. 

 {\it Sufficiency}:
 Let $e=xy$ be a $\gamma_{\mathcal{H}}$-$ER^-$-critical edge of $G$ 
 and $M$ a  $\gamma_\mathcal{H}$-set of $G-e$. Hence 
  $\gamma_\mathcal{H} (G -e) < \gamma_\mathcal{H} (G)$ and $x, y \in M$. 
  But then $M$ is a dominating $\mathcal{H}$-set of $G_e$ which leads to 
 $\gamma_\mathcal{H} (G_e) \leq \gamma_\mathcal{H} (G-e) < \gamma_\mathcal{H} (G)$.
 
{\it Necessity}:  
Let $e=xy$ be a $\gamma_{\mathcal{H}}$-S$^-$-critical edge of $G$ and 
$M$  a $\gamma_\mathcal{H}$-set of $G_e$. Hence $\gamma_\mathcal{H} (G_e) < \gamma_\mathcal{H} (G)$.   
  Assume $v \not \in M$. Hence at least one of $x$ and $y$ is in $M$. If both $x, y \in M$ then $M$ is a dominating $\mathcal{H}$-set of $G-e$ and the result follows. 
 If $x \not \in M$ and $y \in M$ then $M$ is a dominating $\mathcal{H}$-set of $G$, a contradiction. 
 Thus we may assume $v$ is in all $\gamma_\mathcal{H}$-sets of $G_e$.  
  Since $\mathcal{H}$ is induced-hereditary, at least one of $x$ and $y$ is not in $M$. 
  First let $x \in M$ and $y \not \in M$. Then $y \in pn_{G_e}[v,M]$ which implies 
  $M-\{v\}$ is a dominating $\mathcal{H}$-set of $G$ - a contradiction. Hence both $x$ and $y$ are not in $M$. 
  If $x, y \not \in pn_{G_e} [v,M]$ then $M - \{v\}$ is a dominating $\mathcal{H}$-set of $G$, a contradiction.
  Hence  at least one of $x$ and $y$, say $y$, is in $pn_{G_e}[v, M]$. 
	But then $(M - \{v\}) \cup \{y\}$ is a dominating $\mathcal{H}$-set of $G$, a contradiction. 
\end{proof}

Note that (a) there do not exist $\gamma$-$ER^-$-critical edges (see \cite{gm}), 
and (b) necessary and sufficient conditions for an edge of a graph $G$ 
to be $\gamma_{\mathcal{H}}$-$ER^-$-critical may be found in \cite{sam}. 
Now we define  the following classes of graphs:

$\bullet$ $(CS^-_{\mathcal{P}})$  \ 
          $\gamma_{\mathcal{P}}(G) >\gamma_{\mathcal{P}}(G_e)$
          for every edge $e$ of $G$, and 

$\bullet$ $(CER^-_{\mathcal{P}})$ \ 
          $\gamma_{\mathcal{P}}(G) > \gamma_{\mathcal{P}}(G-e)$ 
          for every edge $e$ of $G$.

As an immediate consequence of  Theorem \ref{minus} we obtain the next result. 

\begin{corollary}\label{coin}
If $\mathcal{H} \subseteq \mathcal{I}$ is induced-hereditary and closed under union with $K_1$ 
then the classes of graphs $(CS^-_{\mathcal{P}})$ and $(CER^-_{\mathcal{P}})$ coincide.
\end{corollary}

Note that the class $(CER^-_{\mathcal{P}})$ in the case when $\mathcal{P} = \mathcal{O}$ 
 was introduced by Grobler \cite{g}  and also considered in \cite{gm0,gm,cfm}.

\section{Multiple subdivision}

Recall that  $G_{e,t}$ denote the graph obtained from a graph $G$ 
by subdivision of the edge $e \in E(G)$ with $t$ vertices
(instead of edge $e = uv$ we put a path $(u,x_1,x_2,\ldots ,x_t,v)$). 
For any graph $G$ and any nondegenerate property $\mathcal{P}$ 
let us denote by $V^-_\mathcal{P}(G)$ the set 
$\{v \in V(G) \mid \gamma_{\mathcal{P}}(G-v) < \gamma_{\mathcal{P}}(G)\}$.
Qur first result shows that the value of the difference 
$ \gamma_{\mathcal{P}}(G_{e,3}) - \gamma_{\mathcal{P}}(G-e)$  can be either $0$ or $1$.

\begin{theorem}\label{multi1}
Let $\mathcal{H} \subseteq \mathcal{I}$ be induced-hereditary and closed under union with $K_1$. 
If $e=uv$ is an edge of a graph $G$ then  
$\gamma_{\mathcal{H}}(G-e) \leq \gamma_{\mathcal{H}}(G_{e,3}) \leq  \gamma_{\mathcal{H}}(G-e) + 1$. 
Moreover,  the following conditions are equivalent:
\begin{itemize}
\item[($\mathbb{A}_1$)] $\gamma_{\mathcal{H}}(G-e) = \gamma_{\mathcal{H}}(G_{e,3})$;
\item[($\mathbb{A}_2$)]
                        at least one of the following holds:
\begin{itemize}
\item[(i)] $u \in V^-_\mathcal{H}(G-e)$ and $v$ belongs to some $\gamma_\mathcal{H}$-set of $G-u$; 
\item[(ii)] $v \in V^-_\mathcal{H}(G-e)$ and $u$ belongs to some $\gamma_\mathcal{H}$-set of $G-v$.
\end{itemize}
\end{itemize}
If in addition $\mathcal{H}$ is hereditary then ($\mathbb{A}_1$)  and ($\mathbb{A}_2$)  are equivalent to
\begin{itemize}
\item[($\mathbb{A}_3$)] $\gamma_\mathcal{H}(G-e) = 1 + \gamma_\mathcal{H}(G)$.
\end{itemize}
\end{theorem}

The main result in this section is the following.

\begin{theorem}\label{multi4}
Let $e$ be an edge of a graph $G$ and let 
 $\mathcal{H} \subseteq \mathcal{I}$ be hereditary and closed under union with $K_1$.
 \begin{itemize}
 \item[(i)] Then $\gamma_\mathcal{H}(G) = \gamma_\mathcal{H}(G_{e,3})$ if and only if 
 $\gamma_\mathcal{H}(G) = \gamma_\mathcal{H}(G-e) +1$.
  \item[(ii)] If $\gamma_\mathcal{H}(G) = \gamma_\mathcal{H}(G-e) +1$ then 
  $msd_\mathcal{H}(e) = msd^-_\mathcal{H}(e) = 1$, $msd^+_\mathcal{H}(e) = 6$ and 
$\gamma_\mathcal{H}(G) = \gamma_\mathcal{H}(G_{e,1}) +1 = \gamma_\mathcal{H}(G_{e,2}) +1 
 = \gamma_\mathcal{H}(G_{e,3}) =  \gamma_\mathcal{H}(G_{e,4})  =  \gamma_\mathcal{H}(G_{e,5}) = \gamma_\mathcal{H}(G_{e,6}) - 1.$
 \item[(iii)] Then $msd_{\mathcal{H}}(e) \leq 3$.  In particular,
   (Dettlaff,  Raczek and  Topp \cite{drt} when   $\mathcal{H} = \mathcal{I}$) 
		 $msd_{\mathcal{H}}(G) \leq 3$. 
  \end{itemize} 
\end{theorem}

\begin{example}\label{+1}
It is easy to see that
 if $G = K_{3n_2..n_m}$, where $m \geq 2$ and $n_i \geq 3$ for $2 \leq i \leq m$,
          then $\gamma_{\mathcal{O}}(G) = \gamma_\mathcal{O}(G_{e,3}) = \gamma_\mathcal{O}(G-e) +1 = 3$ 
          for every edge $e$ of $G$.         
	Hence by Theorem \ref{multi4}, 	$msd_{\mathcal{O}}(G) = msd_{\mathcal{O}}^-(G) =1$	and $msd_{\mathcal{O}}^+(G) =6$.				
\end{example}

In view of Theorem \ref{multi4}(iii), we can split the family  of all graphs $G$ into three classes
 with respect to the value of  $msd_{\mathcal{P}}(G)$, 
 where $\mathcal{P} \subseteq \mathcal{I}$ is hereditary and closed under union with $K_1$. 
 We define that a graph $G$ belongs to the class
 $S_\mathcal{P}^i$ whenever  $msd_{\mathcal{P}}(G) = i$, $i \in \{1,2,3\}$. 
 It is straightforward to verify that if 
$k \geq 1$ and  $\mathcal{O} \subseteq \mathcal{P} \subseteq \mathcal{I}$ then 

$\bullet$ $P_{3k}, C_{3k}  \in S_\mathcal{P}^1$;  \  $P_{3k+2}, C_{3k+2}  \in S_\mathcal{P}^2$;  \   
and $P_{3k+1}, C_{3k+1}  \in S_\mathcal{P}^3$.

Thus, none of $S_\mathcal{P}^1, S_\mathcal{P}^2$ and $S_\mathcal{P}^3$ is empty. 

We conclude this part with an open problem. 
\begin{problem} \label{class}
Characterize the graphs belonging to
 $S_\mathcal{P}^i$, or find further properties of such graphs.
\end{problem}

Remark  that Dettlaff, Raczek and Topp   recently characterized all trees belonging to $S^1$ and $S^3$ (see \cite{drt}).

\subsection{Proofs}
For the proofs of  Theorems \ref{multi1} and  \ref{multi1}, we need the following results.

 \begin{them}[\cite{sam}]\label{a}
Let $\mathcal{H} \subseteq \mathcal{I}$ be nondegenerate and closed under union with $K_1$. 
Let $G$ be a graph and  $v \in V(G)$. 
\begin{itemize}
\item[(i)] If $v$ belongs to no $\gamma_{\mathcal{H}}$-set of $G$ then $\gamma_\mathcal{H}(G-v) = \gamma_{\mathcal{H}}(G) $.
\item[(ii)] If $\gamma_\mathcal{H}(G-v) < \gamma_{\mathcal{H}}(G) $ then 
                   $\gamma_\mathcal{H}(G-v) = \gamma_{\mathcal{H}}(G) - 1$. Moreover, 
if $M$ is a $\gamma_{\mathcal{H}}$-set of $G-v$ then $M \cup \{v\}$ is a  $\gamma_{\mathcal{H}}$-set of $G$ 
and  $\{v\} = pn_G[v,M \cup \{v\}]$.
\end{itemize}
\end{them}

\begin{them}[\cite{sam}]\label{edgeadd}
Let $\mathcal{H} \subseteq \mathcal{I}$ be hereditary and closed under union with $K_1$. 
Let $e=uv$ be an edge of a graph $G$. 
 If $\gamma_\mathcal{H}(G) < \gamma_\mathcal{H}(G-e)$ then 
 $\gamma_\mathcal{H}(G) = \gamma_\mathcal{H}(G-e) - 1$. 
 Moreover $\gamma_\mathcal{H}(G) = \gamma_\mathcal{H}(G-e) - 1$ 
 if and only if at least one of the conditions (i) and (ii)  stated in Theorem  \ref{multi1}  holds.
\end{them}

\begin{them}[\cite{sam}]\label{+1l}
Let $e=xy$ be an edge of a graph $G$ and let 
 $\mathcal{H} \subseteq \mathcal{I}$ be hereditary and closed under union with $K_1$.
  If $\gamma_\mathcal{H}(G) > \gamma_\mathcal{H}(G-e)$ then:
	\begin{itemize}
	\item[(i)] no $\gamma_\mathcal{H}$-set of $G-e$ is an $\mathcal{H}$-set of $G$;
	\item[(ii)] both $x$ and $y$ are in all $\gamma_\mathcal{H}$-sets of $G-e$; 
	\item[(iii)] $\gamma_\mathcal{H}(G-x) \geq \gamma_\mathcal{H}(G-e)$ and  $\gamma_\mathcal{H}(G-y) \geq \gamma_\mathcal{H}(G-e)$;
	\item[(iv)] if $\gamma_\mathcal{H}(G-x) = \gamma_\mathcal{H}(G-e)$ then $y$ belongs to no $\gamma_\mathcal{H}$-set of $G-x$;
	\item[(v)] if $\gamma_\mathcal{H}(G-y) = \gamma_\mathcal{H}(G-e)$ then $x$ belongs to no $\gamma_\mathcal{H}$-set of $G-y$.
	\end{itemize}
\end{them}


\begin{proof}[Proof of Theorem \ref{multi1}] 
Let $D$ be a $\gamma_{\mathcal{H}}$-set of $G-e$. 
Then since   $\mathcal{H}$ is closed under union with $K_1$, $D \cup \{x_2\}$ is a dominating $\mathcal{H}$-set of  $G_{e,3}$. 
Hence $\gamma_{\mathcal{H}}(G_{e,3}) \leq |D \cup \{y\}| \leq  \gamma_{\mathcal{H}}(G-e) + 1$. 

 For the left side inequality,  let $\widetilde{D}$ be
  a $\gamma_{\mathcal{H}}$-set of $G_{e,3}$ and $S = \widetilde{D} \cap \{x_1,x_2,x_3\}$. 
If $S=\{x_2\}$ then  $\widetilde{D}- \{x_2\}$ is a dominating $\mathcal{H}$-set of $G-e$ and 
$\gamma_{\mathcal{H}}(G-e) \leq |\widetilde{D}- \{x_2\}| = \gamma_{\mathcal{H}}(G_{e,3}) - 1$. 
If $S = \{x_1,x_2\}$  then $pn_{G_{e,3}}[x_1,\widetilde{D}] = \{u\}$ and hence $\widetilde{D}_1 = (\widetilde{D} - \{x_1,x_2\}) \cup \{u\}$ 
is a dominating $\mathcal{H}$-set of $G-e$  which implies
 $\gamma_{\mathcal{H}}(G-e) \leq |\widetilde{D}_1| < |\widetilde{D}| = \gamma_{\mathcal{H}}(G_{e,3})$.

Let $S = \{x_1\}$. 
If $u \in pn[x_1,\widetilde{D}]$ then $\widetilde{D}_2 = (\widetilde{D} - \{x_1\}) \cup \{u\}$  is a dominating $\mathcal{H}$-set of $G-e$ and 
hence  $\gamma_{\mathcal{H}}(G-e) \leq |\widetilde{D}_2| = |\widetilde{D}| = \gamma_{\mathcal{H}}(G_{e,3})$.
If $u \not\in pn[x_1,\widetilde{D}]$ then $\widetilde{D} - \{x_1\}$  is a dominating $\mathcal{H}$-set of $G-e$ and 
$\gamma_{\mathcal{H}}(G-e) \leq |\widetilde{D}| - 1  = \gamma_{\mathcal{H}}(G_{e,3}) - 1$.

If $S = \{x_1,x_3\}$  then at least one of $pn_{G_{e,3}}[x_1,\widetilde{D}] = \{x_1, u\}$ and $pn_{G_{e,3}}[x_3,\widetilde{D}] = \{x_3, v\}$ holds, 
otherwise $(\widetilde{D} - \{x_1,x_3\}) \cup \{x_2\}$ would be a dominating $\mathcal{H}$-set of $G_{e,3}$, 
contradicting the choice of $\widetilde{D}$.
Say, without loss of generality, $pn_{G_{e,3}}[x_3,\widetilde{D}] = \{x_3, v\}$. 
Then $\widetilde{D}_3 = (\widetilde{D} - \{x_3\}) \cup \{v\}$ is 
a $\gamma_{\mathcal{H}}$-set of $G_{e,3}$
and $\widetilde{D}_3\cap \{x_1,x_2,x_3\} = \{x_1\}$. As above we obtain  
$\gamma_{\mathcal{H}}(G-e) <  \gamma_{\mathcal{H}}(G_{e,3})$.
By reason of symmetry, the left side inequality is proved.
\medskip

($\mathbb{A}_2$) $\Rightarrow$ ($\mathbb{A}_1$)
Assume without loss of generality that (i) holds. 
Let $D$ be a $\gamma_\mathcal{H}(G-u)$-set and $v \in D$. 
By Theorem \ref{a}, $D \cup \{u\}$ is a $\gamma_\mathcal{H}$-set of $G-e$ and $pn_{G-e}[u, D \cup \{u\}] = \{u\}$. 
 Hence $D \cup \{x_1\}$ is a dominating $\mathcal{H}$-set of $G_{e,3}$ and 
$\gamma_{\mathcal{H}}(G_{e,3}) \leq |D \cup \{x_1\}| = \gamma_{\mathcal{H}}(G-e)$. 
 But we have already known that $\gamma_{\mathcal{H}}(G_{e,3}) \geq  \gamma_{\mathcal{H}}(G-e)$.
  Therefore $\gamma_{\mathcal{H}}(G_{e,3})  = \gamma_{\mathcal{H}}(G-e)$.
	\medskip
	
($\mathbb{A}_1$) $\Rightarrow$ ($\mathbb{A}_2$) 
Suppose $\gamma_{\mathcal{H}}(G-e) = \gamma_{\mathcal{H}}(G_{e,3})$. 
	Let $\widetilde{D}$ be a $\gamma_{\mathcal{H}}(G_{e,3})$-set and $S = \widetilde{D} \cap \{x_1,x_2,x_3\}$.  
	If $S=\{x_2\}$ then  $\widetilde{D}- \{x_2\}$ is a dominating $\mathcal{H}$-set of $G-e$, a contradiction. 
	If $S = \{x_1,x_2\}$  then clearly $pn_{G_{e,3}}[x_1,\widetilde{D}] = \{u\}$ which implies that 
	$ (\widetilde{D} - \{x_1,x_2\}) \cup \{u\}$ is a dominating $\mathcal{H}$-set of $G-e$, a contradiction. 
	
Let $S = \{x_1\}$. Hence $v \in \widetilde{D}$.  
If $u \not\in pn_{G_{e,3}}[x_1,\widetilde{D}]$ then $\widetilde{D} - \{x_1\}$ 
 is a dominating $\mathcal{H}$-set of $G-e$, a contradiction. 
	If $u \in pn_{G_{e,3}}[x_1,\widetilde{D}]$ then $D_1 = (\widetilde{D} - \{x_1\}) \cup \{u\}$ is a $\gamma_\mathcal{H}$-set of $G-e$, 
	$u, v \in D_1$, $D_1 - \{u\}$ is a $\gamma_{\mathcal{H}}$-set of $G-u$ (by Theorem \ref{a}) and $v \in D_1 - \{u\}$. 
	In addition it follows that $u \in V^-_\mathcal{H}(G-e)$.  Thus, (i) holds.

By symmetry it remains the case when  $S = \{x_1,x_3\}$.  
If $u \not\in pn_{G_{e,3}}[x_1, \widetilde{D}]$ and $v\not\in pn_{G_{e,3}}[x_3, \widetilde{D}]$ then 
$\widetilde{D} - \{x_1,x_3\}$  is a dominating $\mathcal{H}$-set of $G-e$, a contradiction. 
If $u \in pn_{G_{e,3}}[x_1, \widetilde{D}]$ and $v\not\in pn_{G_{e,3}}[x_3, \widetilde{D}]$ then 
$(\widetilde{D} - \{x_1,x_3\}) \cup \{u\}$ is a dominating $\mathcal{H}$-set of $G-e$, a contradiction. 
So, $u \in pn_{G_{e,3}}[x_1, \widetilde{D}]$ and $v\in pn_{G_{e,3}}[x_3, \widetilde{D}]$. 
Then $D_2 = (\widetilde{D} - \{x_1,x_3\}) \cup \{u, v\}$  is a $\gamma_\mathcal{H}$-set of $G-e$ and 
both $\{u\} = pn_{G-e}[x_1, D_2]$ and $\{v\} = pn_{G-e}[x_3, D_2]$ hold.
  Thus both (i) and (ii) are fulfilled. 
\medskip
  
($\mathbb{A}_2$) $\Leftrightarrow$ ($\mathbb{A}_3$)
  By Theorem \ref{edgeadd}.
\end{proof}
\medskip

\begin{proof}[Proof of Theorem \ref{multi4}] 
(i)
 {\it Necessity}: Let $\gamma_\mathcal{H}(G) = \gamma_\mathcal{H}(G_{e,3})$. 
By Theorem \ref{multi1} we know that 
$\gamma_{\mathcal{H}}(G-e) \leq \gamma_{\mathcal{H}}(G_{e,3}) \leq  \gamma_{\mathcal{H}}(G-e) + 1$  
 and if $\gamma_{\mathcal{H}}(G-e) = \gamma_{\mathcal{H}}(G_{e,3})$ then  
 $\gamma_{\mathcal{H}}(G_{e,3}) =  \gamma_{\mathcal{H}}(G) + 1$. 
 Thus $\gamma_\mathcal{H}(G) = \gamma_\mathcal{H}(G_{e,3}) = \gamma_\mathcal{H}(G-e) +1$. 
 \medskip

{\it Sufficiency}:
 Let $\gamma_{\mathcal{H}}(G-e) + 1 = \gamma_{\mathcal{H}}(G)$.
 Assume $\gamma_\mathcal{H}(G) \not= \gamma_\mathcal{H}(G_{e,3})$. 
 Now by Theorem \ref{multi1}, $\gamma_\mathcal{H}(G_{e,3}) = \gamma_\mathcal{H}(G-e)$. 
 Applying again Theorem \ref{multi1} we obtain 
 $\gamma_\mathcal{H}(G) = \gamma_\mathcal{H}(G-e) -1$, a contradiction. 
 Thus, $\gamma_\mathcal{H}(G) = \gamma_\mathcal{H}(G_{e,3})$.
\newpage

 (ii) By (i), $\gamma_\mathcal{H}(G) = \gamma_\mathcal{H}(G_{e,3})$. 
Let $M$ be a $\gamma_\mathcal{H}$-set  of $G-e$ and $e=uv$.  By Theorem \ref{+1l}(ii), both $u$ and $v$ are in $M$. 
 Then (a) $M$ is a dominating $\mathcal{H}$-set of $G_{e,1}$ and $G_{e,2}$, 
(b) $M \cup \{x_3\}$ is  a dominating $\mathcal{H}$-set of $G_{e,4}$ and $G_{e,5}$, and 
(c) $M \cup \{x_3, x_5\}$ is  a dominating $\mathcal{H}$-set of $G_{e,6}$.  Hence 
\begin{itemize}
\item[(A)] $\gamma_\mathcal{H}(G_{e,i})  \leq \gamma_\mathcal{H}(G-e) = \gamma_\mathcal{H}(G) -1$ for $i = 1, 2$;
                     $\gamma_\mathcal{H}(G_{e,j})  \leq \gamma_\mathcal{H}(G-e)  + 1 = \gamma_\mathcal{H}(G)$ for $i = 4,5$; 
                      $\gamma_\mathcal{H}(G_{e,6})  \leq \gamma_\mathcal{H}(G-e) + 2 = \gamma_\mathcal{H}(G) + 1$.
\end{itemize}
By Theorem \ref{+1l}, $\min\{\gamma_\mathcal{H}(G-u), \gamma_\mathcal{H}(G-v)\} \geq \gamma_\mathcal{H}(G-e)$ 
and by Theorem \ref{a} we have 
$\gamma_\mathcal{H}(G-\{u,v\}) = \gamma_\mathcal{H}((G-u) - v) \geq \gamma_\mathcal{H}(G-u) - 1 \geq \gamma_\mathcal{H}(G-e) -1$.
 Suppose that $\gamma_\mathcal{H}(G-\{u,v\}) = \gamma_\mathcal{H}(G-e) -1$. 
 Then both $\gamma_\mathcal{H}(G-u) = \gamma_\mathcal{H}(G-e)$ and 
$\gamma_\mathcal{H}((G-u) - v) = \gamma_\mathcal{H}(G-u) - 1$ hold. 
By the second equality and Theorem \ref{a} we deduce that 
$v$ belongs to some $\gamma_{\mathcal{H}}$-set of $G-u$. 
 On the other hand, since 
 $\gamma_\mathcal{H}(G) = \gamma_\mathcal{H}(G-e) + 1 > \gamma_\mathcal{H}(G-u)$, 
 $v$ belongs to no $\gamma_\mathcal{H}$-set of $G-u$, a contradiction.
 Thus, 
\begin{itemize}
\item[(B)]  $\min\{\gamma_\mathcal{H}(G-u), \gamma_\mathcal{H}(G-v), \gamma_\mathcal{H}(G-\{u, v\}) \} 
                       \geq \gamma_\mathcal{H}(G-e)$.
\end{itemize}
Let $D_t$ be a $\gamma_\mathcal{H}$-set of $G_{e,t}$ and $U_t = D_t \cap \{x_1,\dots,x_t\}$, where $t = 1,\dots,6$. 

{\it Case} 1: $t \in \{1,2\}$. 

Assume $U_t \not = \emptyset$. 
Then $D_t - U_t$ is a dominating $\mathcal{H}$-set for at least one of the graphs 
$G-e$, $G-u$, $G-v$ and $G - \{u,v\}$.  Using $(B)$ we have 
\[
\gamma_\mathcal{H}(G) = \gamma_\mathcal{H}(G-e) + 1  \leq |D_t - U_t| +1 = \gamma_\mathcal{H} (G_{e,t}) - |U_t| +1 \leq \gamma_\mathcal{H} (G_{e,t}),
\]
contradicting $(A)$.  Thus $U_t$ is empty. 
But then $D_t$ is a  dominating $\mathcal{H}$-set of $G-e$  which leads to 
$\gamma_\mathcal{H} (G_{e,t}) \geq  \gamma_\mathcal{H} (G-e)$.   
Now by $(A)$ it follows the  equality $\gamma_\mathcal{H} (G_{e,t}) =  \gamma_\mathcal{H} (G-e)$. 

{\it Case} 2: $t \in \{4,5\}$.  

Obviously $U_t \not = \emptyset$. As in Case 1 we obtain $\gamma_\mathcal{H}(G) \leq \gamma_\mathcal{H} (G_{e,t})$. 
Since by $(A)$ $\gamma_\mathcal{H}(G_{e,t})  \leq  \gamma_\mathcal{H}(G)$, we have 
$\gamma_\mathcal{H}(G_{e,t})  =  \gamma_\mathcal{H}(G)$.

{\it Case} 3: $t =6$.  

Clearly $|U_6| \geq 2$. 
As in Case 1 we obtain  $\gamma_\mathcal{H}(G) \leq \gamma_\mathcal{H} (G_{e,6}) - |U_6| +1$. 
Since $|U_6| \geq 2$,  $\gamma_\mathcal{H}(G) \leq \gamma_\mathcal{H} (G_{e,6}) - 1$. 
Now by $(A)$ we deduce that $\gamma_\mathcal{H}(G) = \gamma_\mathcal{H} (G_{e,6}) - 1$. 
\medskip

(iii) Immediately by (i) and (ii).
\end{proof}

\end{document}